\documentclass[12pt]{amsart}
\usepackage{amsmath,amsfonts,amssymb,latexsym}
\usepackage{hyperref}
\usepackage{graphicx}
\usepackage[a4paper, left=2cm, right=2cm, top=2cm, bottom=2cm]{geometry}
\usepackage[utf8]{inputenc}
\usepackage{microtype}
\usepackage{graphicx}

\newcommand{\NN}{\mathbb{N}}
\newcommand{\CC}{\mathbb{C}}
\newcommand{\RR}{\mathbb{R}}
\newcommand{\EE}{\mathbb{E}}
\newcommand{\Oo}{\mathcal{O}}
\newcommand{\Bo}{\mathcal{B}}
\newcommand{\MM}{\mathcal{M}}
\newcommand{\ord}{\textrm{ord}_t}

\renewcommand{\Re}{\operatorname{Re}}

\renewcommand{\epsilon}{\varepsilon}
\renewcommand{\phi}{\varphi}

\allowdisplaybreaks

\newtheorem{thm}{Theorem}[section]
\newtheorem{lemma}{Lemma}[section]
\newtheorem{prop}{Proposition}[section]
\theoremstyle{definition}
\newtheorem{remark}{Remark}[section]
\newtheorem{defi}{Definition}[section]
\newtheorem{example}{Example}[section]
\usepackage{enumerate}

\begin{document}
\title[Gevrey estimates for certain moment pdes]{Gevrey estimates for certain moment partial differential equations}
  \author{S{\l}awomir Michalik}
  \address{Faculty of Mathematics and Natural Sciences,
College of Science\\
Cardinal Stefan Wyszy\'nski University\\
W\'oycickiego 1/3,
01-938 Warszawa, Poland\\
ORCiD: 0000-0003-4045-9548}
\email{s.michalik@uksw.edu.pl}
\urladdr{\url{http://www.impan.pl/~slawek}}
\author{Maria Suwi\'nska}
\address{Faculty of Mathematics and Natural Sciences, College of Science\\ Cardinal Stefan Wyszy\'nski University\\
W\'oycickiego 1/3, 01-938 Warszawa, Poland}
\email{m.suwinska@op.pl}
\begin{abstract}
 We consider the Cauchy problem for  inhomogeneous 
 linear moment differential equations with holomorphic
 time dependent coefficients. Using such tools as the
 formal norms, theory of majorants and the properties of
 the Newton polygon, we obtain the Gevrey estimate for
 the formal solution of the equation.
\end{abstract}

\maketitle

\section{Introduction}

The concept of $m$-moment differentiation $\partial_{m,t}$ is a
generalization of, among other differential operators, standard differentiation.
More precisely, if we consider a formal power series $\hat{u}(t)=\sum_{n=0}^{\infty}\frac{u_{n}}{m(n)}t^{n}$,
where $m(n)$ is a moment function (see: Definition \ref{defi:kernel})
then 
\begin{equation*}
\partial_{m,t}\hat{u}(t)=\partial_{m,t}\sum_{n=0}^{\infty}\frac{u_{n}}{m(n)}t^{n}=\sum_{n=0}^{\infty}\frac{u_{n+1}}{m(n)}t^{n}.
\end{equation*}

Moment partial differential equations emerged in \cite{BY} by W.
Balser and M. Yoshino. Later S.~Michalik analysed properties of analytic
solutions of linear moment PDEs with constant coefficients in \cite{Mic1}
and \cite{Mic2}.

In this paper we consider the initial value problem for a linear moment
partial differential equation in the variables $t\in\CC$ and $z=(z_{1},\ldots,z_{N})\in\CC^{N}$
of the form 
\begin{equation}
\left\{ \begin{aligned}\partial_{m_{0},t}^{M}u(t,z)+\sum_{(j,\alpha)\in\Lambda}a_{j,\alpha}(t)\partial_{m_{0},t}^{j}\partial_{m,z}^{\alpha}u(t,z) & =f(t,z)\\
\partial_{m_{0},t}^{j}u(0,z) & =\phi_{j}(z)\textrm{ for }0\leq j<M
\end{aligned}
\right.,\label{eq:main}
\end{equation}
where $\partial_{m_{0},t}$ and $\partial_{m,z}^{\alpha}:=\partial_{m_{1},z_{1}}^{\alpha_{1}}\partial_{m_{2},z_{2}}^{\alpha_{2}}\dots\partial_{m_{N},z_{N}}^{\alpha_{N}}$
denote moment differential operators and the coefficients $a_{j,\alpha}(t)$
do not depend on the variable $z$.

Our main goal is to find the Gevrey estimate of the formal solution
$\hat{u}(t,z)=\sum_{n=0}^{\infty}u_{n}(z)t^{n}$ of (\ref{eq:main})
under a certain set of conditions (see Section \ref{s:main_problem}
for more details). To achieve that we generalise some of the results presented
by H. Tahara and H. Yamazawa in \cite{TY}. In particular, we focus on Theorem 5.1 from the aforementioned paper concerning the Gevrey estimate for solutions of linear PDEs with coefficients depending only on the variable $t$.
\par 
Similar result for the Gevrey order of the formal solutions of the Cauchy problem to linear partial differential equations with variable coefficients was given by A.~Yonemura
\cite{Y}. His result was generalised in many directions.
The extension to different type of Cauchy-Goursat problems for linear equations were given by M.~Miyake \cite{M}, M.~Miyake and Y.~Hashimoto \cite{MH}, M.~Miyake and M.~Yoshino \cite{MY}. Similar problems for nonlinear partial differential equations were studied by such authors as
S.~\={O}uchi \cite{O}, R.~G{\'e}rard and H.~Tahara \cite{GT1,GT2} and A.~Shirai \cite{S}.
\bigskip
\par
Throughout this paper the following notation will be used.

By $D^N_{r}$ we denote the open polydisc in $\CC^{N}$ with radius $r>0$
and a center at the origin, i.e. 
$$D^N_{r}=\{(z_1,\dots,z_N)\in\CC^{N}\colon |z_j|<r\ \textrm{for}\ j=1,\dots,N\}.
$$
If $N=1$ we denote it shortly by $D_{r}$.

For any $d\in\RR$ and $\epsilon>0$ a set 
\begin{equation*}
S_{d}(\epsilon)=\{z\in\CC\colon |z|>0,\,d-\epsilon/2<\arg(z)<d+\epsilon/2\}
\end{equation*}
will be called a sector in a direction $d$ with an opening $\epsilon$
on $\CC$.

If a function $f$ is holomorphic on a set $G\subset\CC^{N}$ then
we will write that $f\in\Oo(G)$.
More generally, if $\EE$ denotes a complex Banach space with a norm $\|\cdot\|_{\EE}$, then
by $\Oo(G,\EE)$ we shall denote the set of all $\EE$-valued  
holomorphic functions on a set $G\subseteq\CC^N$.
For more information about functions with values in Banach spaces we refer the reader to \cite[Appendix B]{B}. 
In the paper, as a Banach space $\EE$ we will
take mainly the space of complex numbers $\CC$ (we abbreviate $\Oo(G,\CC)$ to $\Oo(G)$)
or the space of functions $\Oo_{r}:=\Oo(D_r^N)\cap C(\overline{D_r^N})$ equipped with the norm
$\|\varphi\|_{\Oo_{r}}:=\sup_{z\in D_r^N}|\varphi(z)|$.

For a function $f(t)\in\Oo(D_r)$ we
denote by $\ord(f)$ the order of zero of the function $f(t)$ at $t=0$.

We will denote a space of formal power series $\hat{u}(t)=\sum_{n=0}^{\infty}u_{n}t^{n}$
with coefficients from any Banach space $\mathbb{E}$ by $\mathbb{E}[[t]]$.

Let $\sum_{\alpha\in\NN_0^N}a_{\alpha}x^{\alpha}$ and
$\sum_{\alpha\in\NN_0^N}b_{\alpha}x^{\alpha}$
be two formal power series. We write that $\sum_{\alpha\in\NN_0^N}a_{\alpha}x^{\alpha}\ll\sum_{\alpha\in\NN_0^N}b_{\alpha}x^{\alpha}$,
when $|a_{\alpha}|\leq b_{\alpha}$ for all $\alpha\in\NN_{0}^N$. Then we also call
$\sum_{\alpha\in\NN_0^N}b_{\alpha}x^{\alpha}$ a \emph{majorant} of $\sum_{\alpha\in\NN_0^N}a_{\alpha}x^{\alpha}$.

\section{Moment functions and moment differential operators}

In this section the basic theory of moment functions introduced by W. Balser \cite{B} and moment differential operators defined by W. Balser and
M. Yoshino \cite{BY} will be recalled.

\begin{defi}[compare {\cite[Section 5.5]{B}}]\label{defi:kernel} A pair of functions
$e_{m}$ and $E_{m}$ is said to be \emph{kernel functions of order}
$k>1/2$ if: \begin{enumerate} 

\item $e_{m}\in\Oo(S_{0}(\pi/k))$, $e_{m}(z)/z$ is integrable at
the origin, $e_{m}(x)\in\RR_{+}$ for $x\in\RR_{+}$ and for any $\epsilon>0$
there exist constants $A,B>0$ such that $|e_{m}(z)|\leq Ae^{-(|z|/B)^{k}}$
for $z\in S_{0}(\pi/k-\varepsilon)$, 

\item $E_{m}\in\Oo(\CC)$ and there exist $A,B>0$ such that $|E_{m}(z)|\leq Ae^{B|z|^{k}}$
for $z\in\CC$ and $E_{m}(1/z)/z$ is integrable at the origin in
$S_{\pi}(2\pi-\pi/k)$. 

\item The connection between the above functions $e_{m}$ and $E_{m}$
is given by \emph{the corresponding moment function $m$ of order
$1/k$} defined as 
\begin{equation*}
m(u):=\int_{0}^{\infty}x^{u-1}e_{m}(x)dx
\end{equation*}
for all $u$ such that $\Re u\geq0$, and the kernel function $E_{m}$
has the power series expansion 
\begin{equation*}
E_{m}(z)=\sum_{n=0}^{\infty}\frac{z^{n}}{m(n)}\ \ \textsf{for}\ \ z\in\CC.
\end{equation*}

\item Additionally we shall assume that the normalisation property
holds for the corresponding moment function $m$, i.e. that $m(0)=1$.

\end{enumerate} \end{defi}

Please note that for $k\leq1/2$ the sector $S_{\pi}(2\pi-\pi/k)$
is not defined. For that case the kernel functions as well as their
corresponding moment function have to be defined separately.

\begin{defi}[see {\cite[Section 5.6]{B}}] For any $k>0$ we may define a \emph{kernel function
$e_{m}$ of order $k$} if there is $n\in\NN$ such that $nk>1/2$
and there exist kernel functions $e_{\tilde{m}}$ and $E_{\tilde{m}}$
of order $nk$ satisfying the following condition: 
\begin{equation*}
e_{m}(z)=\frac{e_{\tilde{m}}(z^{1/n})}{n}.
\end{equation*}
Then both the kernel function $E_{m}$ of order $k$ and the corresponding
moment function $m$ of order $1/k$ are defined by the same formulas
as in Definition \ref{defi:kernel}.
\end{defi}

\begin{example}
\label{ex:functions}
 For any $k>0$ the classical and most important kernel functions of order $k$ and the corresponding moment function of order $1/k$ are given by
 \begin{itemize}
  \item $e_m(z)=kz^{k}e^{-z^k}$;
  \item $m(u)=\Gamma(1+u/k)$, where $\Gamma$ is the Gamma function;
  \item $E_m(z)=\sum_{j=0}^{\infty}\frac{z^j}{\Gamma(1+j/k)}=:\mathbf{E}_{1/k}(z)$, where $\mathbf{E}_{1/k}$ is the
    Mittag-Leffler function of index $1/k$.
\end{itemize}
They are used in the classical theory of $k$-summability.
\end{example}

Moreover the set of moment functions is closed under multiplication and division. Namely, we have
\begin{prop}[see {\cite[Theorems 31 and 32]{B}}] Let $m_{1}$ and $m_{2}$ be two moment
functions of orders $s_{1}$ and $s_{2}$, respectively. Then $m_{1}\cdot m_{2}$ is a moment function
of order $s_{1}+s_{2}$ and if moreover $s_1>s_2$, $\frac{m_{1}}{m_{2}}$ is a moment function
of order $s_{1}-s_{2}.$
\end{prop}
The above proposition suggests to define moment functions of order zero.
\begin{defi}
We call $m$ a \emph{moment function of order $0$},
if there exist moment functions $m_{1},\,m_{2}$ of the same order
$s$ such that $m=\frac{m_{1}}{m_{2}}.$
\end{defi}

It is worth noting that all moment functions of order $s$ have
the same growth as $\Gamma_{s}(x):=\Gamma(1+sx)$ (see \cite[Section 5.5]{B}), i.e. there exist
positive constants $a$ and $A$ such that 
\begin{gather}
\label{eq:gevrey}
a^{n}\Gamma_{s}(n)\leq m(n)\leq A^{n}\Gamma_{s}(n)\quad\textrm{for every}\quad
n\in\NN_0.
\end{gather}

We will consider a special class of moment functions:
\begin{defi}
For any $s>0$ we say that a moment function $m$ of order $s$ is a \emph{regular moment function of order $s$} if there exist constants $a,A>0$ such that 
\begin{equation}
 an^{s}\leq\frac{m(n)}{m(n-1)}\leq An^{s}\quad\textrm{for every}\quad n\in\NN.
\label{de:MM_class}
\end{equation}
We denote the set of regular moment functions of order $s$ by $\MM_s$.
\end{defi}

Regular moment functions satisfy the following properties:
\begin{lemma} \begin{enumerate}[(a)] 
\item The class of regular moment functions is closed under multiplication and
division. More precisely, if $m_{1}\in\MM_{s_{1}}$ and $m_{2}\in\MM_{s_{2}}$ then $m_{1}\cdot m_{2}\in\MM_{s_{1}+s_{2}}$
and if moreover $s_{1}> s_{2}$, $\frac{m_{1}}{m_{2}}\in\MM_{s_{1}-s_{2}}$.\label{enu:moment_class1} 

\item The class of regular moment functions contains classical moment functions $\Gamma_s$ of order $s$, i.e. $\Gamma_{s}\in\MM_{s}$ for any $s>0$.\label{enu:moment_class2}
\end{enumerate}
\end{lemma}
\begin{proof}
Let $a_{1},A_{1},a_{2},A_{2}$
be positive constants such that for $i=1,2$ and for every $n\in\NN$ we have: 
\begin{equation*}
a_{i}n^{s_{i}}\leq\frac{m_{i}(n)}{m_{i}(n-1)}\leq A_{i}n^{s_{i}}.
\end{equation*}
For any $n\in\NN$ we have: 
\begin{equation*}
a_{1}a_{2}n^{s_{1}+s_{2}}\leq\frac{m_{1}(n)m_{2}(n)}{m_{1}(n-1)m_{2}(n-1)}\leq A_{1}A_{2}n^{s_{1}+s_{2}}.
\end{equation*}

Seeing as $m_{1}\cdot m_{2}$ is a moment function of order $s_{1}+s_{2}$,
this proves the first part of (\ref{enu:moment_class1}). In the case
of $\frac{m_{1}}{m_{2}}$, which is a moment function of order $s_{1}-s_{2}$,
the proof is similar and it will be omitted.

In order to show (\ref{enu:moment_class2}) we can use the Stirling formula (see \cite{WW}): 
\begin{equation}
\label{eq:stirling}
\sqrt{2\pi}x^{x-\frac{1}{2}}e^{-x}\leq\Gamma(x)\leq\sqrt{2\pi}e^{\frac{1}{12x}}x^{x-\frac{1}{2}}e^{-x}<\sqrt{2\pi}x^{x-\frac{1}{2}}e^{-x+1}\textrm{ for }x\geq1.
\end{equation}
Then for any $n\in\NN$ we receive: 
\begin{align*}
\frac{\Gamma(1+ns)}{\Gamma(1+ns-s)} & \leq e^{-s+1}\left(\frac{1+ns}{1+ns-s}\right)^{1+ns-s-\frac{1}{2}}(1+ns)^{s}\\
 & \leq e^{-s+1}\left(1+\frac{s}{1+ns-s}\right)^{1+ns-s}\sqrt{\frac{1+ns-s}{1+ns}}\left(1+\frac{1}{s}\right)^{s}s^{s}n^{s}\\
 & \leq \left(1+\frac{1}{s}\right)^{s}es^{s}n^{s}.
\end{align*}

In a similar
way we obtain the inequalities: 
\begin{align*}
\frac{\Gamma(1+ns)}{\Gamma(1+ns-s)} & \geq e^{-s-1}\left(\frac{1+ns}{1+ns-s}\right)^{ns-s+\frac{1}{2}}(1+ns)^{s}\\
 & \geq e^{-s-1}s^{s}n^{s}.
\end{align*} 
\end{proof}

One more property of the Gamma function will be used extensively in
subsequent sections of this paper:

\begin{lemma}\label{lem:gamma}
For any $s\geq0$ there exist constants $\tilde{c_{s}}$ and $\tilde{C}_{s}$
such that for all $x\geq s$ we have
\begin{equation*}
\tilde{c}_{s}(1+x)^{s}\leq\frac{\Gamma(1+x)}{\Gamma(1+x-s)}\leq\tilde{C}_{s}(1+x)^s.
\end{equation*}
\end{lemma}

\begin{proof}
As before, we use the Stirling formula (\ref{eq:stirling}) to receive:
\begin{equation*}
\frac{\Gamma(1+x)}{\Gamma(1+x-s)}\leq\left(\frac{1+x}{1+x-s}\right)^{1/2+x-s}e^{-s+1}(1+x)^{s}\leq e(1+x)^{s}.
\end{equation*}
Similarly, we show the second inequality:
\begin{equation*}
\frac{\Gamma(1+x)}{\Gamma(1+x-s)}\geq\left(\frac{1+x}{1+x-s}\right)^{1/2+x-s}e^{-s-1}(1+x)^{s}\geq e^{-s-1}(1+x)^{s}.
\end{equation*}
\end{proof}

Using moment functions, we can also define a moment Borel transform.

\begin{defi}
Let $m$ be a moment function of order $s$. Then
we define an \emph{$m$-moment Borel transform} as an operator $\Bo_{m,t}\colon\mathbb{E}[[t]]\to\mathbb{E}[[t]]$
given by the formula: 
\begin{equation*}
\Bo_{m,t}\left(\sum_{n=0}^{\infty}u_{n}t^{n}\right):=\sum_{n=0}^{\infty}\frac{u_{n}}{m(n)}t^{n}.
\end{equation*}
\end{defi}

\begin{defi}
Assume that $m$ is a moment function of order
$s\geq 0$. Then $\hat{u}\in\mathbb{E}[[t]]$ is a \emph{formal power series of
Gevrey order $s$} if there exists $r>0$ such that $\Bo_{m,t}\hat{u}\in\Oo(D_r,\EE)$.
We denote the space of all such power series by $\mathbb{E}[[t]]_{s}$.
\end{defi}

\begin{remark}
\label{rem:gevrey}
 Observe that by (\ref{eq:gevrey})  the formal series $\hat{u}=\sum_{n=0}^{\infty}u_nt^n\in\mathbb{E}[[t]]$ is of Gevrey order $s$
 if and only if there exist $B,C>0$ such that
 \begin{gather*}
  \|u_n\|_{\EE}\leq BC^n\Gamma_s(n)\quad\textrm{for}\quad n\in\NN_0.
 \end{gather*}
 For this reason any formal power series of Gevrey order $0$ is convergent.
 
 It also means that the definition of formal power series of Gevrey order $s$ does not depend on the choice of a moment function $m$ of order $s$.
\end{remark} 

\begin{defi}[see \cite{BY}] Let $m$ be a moment function. Then
we define an \emph{$m$-differential operator} $\partial_{m,t}\colon\mathbb{E}[[t]]\to\mathbb{E}[[t]]$
by the formula: 
\begin{equation*}
\partial_{m,t}\left(\sum_{n=0}^{\infty}\frac{u_{n}}{m(n)}t^{n}\right):=\sum_{n=0}^{\infty}\frac{u_{n+1}}{m(n)}t^{n}.
\end{equation*}
\end{defi}

   Below we present most important examples of moment differential operators.
     Other examples, including also integro-differential operators, can be found in
     \cite[Example 3]{Mic3}.
     \begin{example} 
     If $m(u)=\Gamma_1(u)$ then the operator $\partial_{m,x}$ coincides with the usual differentiation $\partial_x$.
     
     More generally, if $s>0$ and $m(u)=\Gamma_s(u)$ then the operator $\partial_{m,x}$ satisfies
     $(\partial_{m,x}\widehat{u})(x^s)=\partial^s_x(\widehat{u}(x^s))$,
     where $\partial^s_x$ denotes the Caputo fractional derivative of order $s$
     defined by
       $$
     \partial^{s}_{x}\Big(\sum_{j=0}^{\infty}\frac{u_{j}}{\Gamma_s(j)}x^{sj}\Big):=
     \sum_{j=0}^{\infty}\frac{u_{j+1}}{\Gamma_s(j)}x^{sj}.$$
     \end{example}

Immediately by the definitions, we obtain the following commutation formula between moment differential operators and moment Borel transforms

\begin{prop}[Commutation formula]
\label{prop:commutation}
Let $m$ and $m'$ be moment functions. Then
the operators $\hat{\Bo}_{m',t}, \partial_{m,t}\colon\EE[[t]]\to\EE[[t]]$ satisfy the commutation formula
$$\hat{\Bo}_{m',t}\partial_{m,t}\hat{u}=\partial_{mm',t}\hat{\Bo}_{m',t}\hat{u}\quad\textrm{for any}\quad \hat{u}\in\EE[[t]].$$
\end{prop}

We estimate moment derivatives of holomorphic functions as follows

\begin{prop}[see {\cite[Lemma 1]{Mic1}}]\label{prop_diff} Let
$f$ be a function holomorphic on $D_{R}\subset\CC$ and let $m$
be a moment function of order $s$. Then for any positive
$r<r'<R$ there
exist a constant $h>0$ such that for all $\alpha\in\NN_{0}$
\begin{equation*}
\sup_{z\in D_{r}}|\partial_{m,z}^{\alpha}f(z)|\leq \sup_{z\in D_{r'}}|f(z)|h^{\alpha}\Gamma(1+s\alpha).
\end{equation*}
\end{prop}

Proposition \ref{prop_diff} can be generalized to the multidimensional case.
\begin{prop}
\label{prop:general}
Let $m_1,\ldots,m_N$ be moment functions of orders $s_1,\ldots,s_N$, respectively, with $s=(s_1,\ldots,s_N)$, and suppose that $f\in\Oo(D^N_R)$ for certain $R>0$. Then for any $0<r<r'<R$ there
exist constant $\tilde{h}>0$ such that for all $\alpha\in\NN^N_{0}$
\begin{equation*}
\sup_{z\in D^N_{r}}|\partial_{m_1,z_1}^{\alpha_1}\ldots\partial_{m_N,z_N}^{\alpha_N}f(z)|\leq 
\sup_{z\in D^N_{r'}}|f(z)|\tilde{h}^{|\alpha|}\Gamma(1+s\cdot\alpha).
\end{equation*}
\end{prop}
\begin{proof}
After applying $N$ times Proposition \ref{prop_diff} we receive:
\begin{equation*}
\sup_{z\in D^N_{r}}|\partial_{m_1,z_1}^{\alpha_1}\ldots\partial_{m_N,z_N}^{\alpha_N}f(z)|\leq \sup_{z\in D^N_{r'}}|f(z)| h^{|\alpha|}\Gamma(1+s_1\alpha_1)\ldots\Gamma(1+s_N\alpha_N).
\end{equation*}
Moreover, let us note that for any complex $u$, $w$ such that $\Re u\geq 0$ and $\Re w \geq 0$ we have $\Gamma(u)\Gamma(w)\leq\Gamma(u+w)$. Using that fact we receive:
\begin{equation*}
\Gamma(1+s_1\alpha_1)\ldots\Gamma(1+s_N\alpha_N)\leq\Gamma(N+s\cdot\alpha).
\end{equation*}
Using properties of the Gamma function we obtain the following:
\begin{align*}
\Gamma(N+s\cdot\alpha)&=(N-1+s\cdot\alpha)\ldots(1+s\cdot\alpha)\Gamma(1+s\cdot\alpha)\\
 &\leq e^{(N-1)(N-2)/2}e^{(N-2)s\cdot\alpha}\Gamma(1+s\cdot\alpha).
\end{align*}
Let us now use a notation $\bar{s}=\max_{1\leq j\leq N}s_j$. Then
\begin{equation*}
\Gamma(N+s\cdot\alpha)\leq e^{(N-1)(N-2)/2}e^{(N-2)\bar{s}|\alpha|}\Gamma(1+s\cdot\alpha).
\end{equation*}
Since we may assume that $|\alpha|\geq 1$, It is therefore enough to take $\tilde{h}=he^{(N-1)(N-2)/2}e^{(N-2)\bar{s}}$.
\end{proof}

\section{Newton polygon}

\label{s:newton_polygon}
The Newton polygon for linear partial differential operators with variable coefficients was introduced by
Yonemura \cite{Y}, who also described the Gevrey order of solution in terms
 of its Newton polygon. In this way he generalized the previous results of \cite{R} given for ordinary differential equations.
 
 On the other hand the Newton polygon for linear moment partial differential
 operators with constant coefficients in two variables  $(t,z)$ was introduced by the first author in \cite{Mic2}.
 
In this section we extend the notion of the Newton polygon to the case of linear moment partial differential operators with variable coefficients and with multidimensional spatial variable $z\in\CC^N$.

 \begin{figure}[hbt]
\includegraphics[width=12cm]{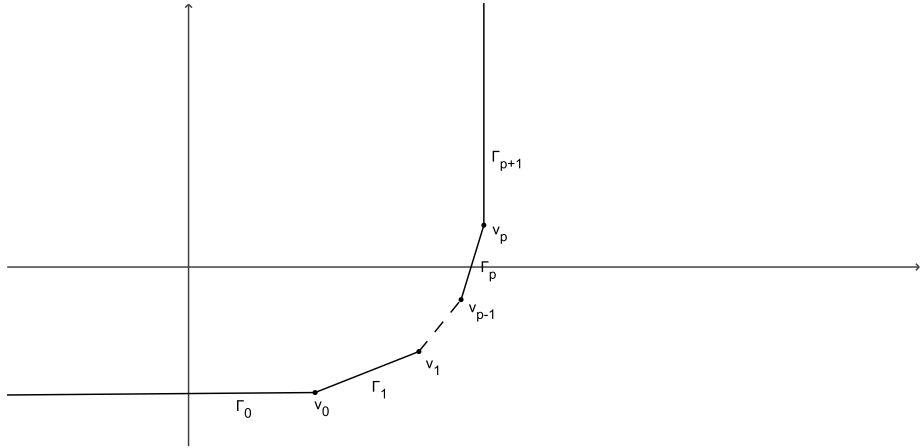} \caption{An example of a Newton polygon.\label{gfx:newton_polygon}}
\end{figure}

Let $m_0,m_1,\dots,m_N$ be moment functions of positive orders $s_0,s_1,\dots,s_N$ respectively.
 We assume that $m=(m_1,\dots,m_N)$, $s=(s_1,\dots,s_N)$, $t\in\CC$, $z=(z_1,\dots,z_N)\in\CC^N$, $\alpha=(\alpha_1,\dots,\alpha_N)\in\NN_0^N$,
 $\Sigma\subseteq\NN_0$ is a set of indices (finite or infinite), $J\subset \NN_0$ and $A\subset \NN_0^N$ are finite set of indices,
 and the moment operator is given by
\begin{equation}
\label{eq:P_op}
 P(t,z,\partial_{m_0,t},\partial_{m,z})=
 \sum_{(\sigma,j,\alpha)\in \Sigma\times J\times A}a_{\sigma j\alpha}(z)
t^{\sigma}\partial_{m_0,t}^j\partial_{m,z}^{\alpha},
\end{equation}
 where we use the multidimensional notation $\partial_{m,z}^{\alpha}:=\partial_{m_1,z_1}^{\alpha_1}\cdots\partial_{m_N,z_N}^{\alpha_N}$.
 
 \begin{defi}
  The \emph{Newton polygon for the operator $P$} given by (\ref{eq:P_op})
  is defined as the convex hull of the union
 of sets
 $Q(js_0+\alpha\cdot s,\sigma-j)$, where 
 $(\sigma,j,\alpha)\in\Sigma\times J\times A$, that is
  \begin{equation*}
N(P,s_0,s)=\textrm{conv}\big\{Q(js_0+\alpha\cdot s,\sigma-j)\colon (\sigma,j,\alpha)\in \Sigma\times J\times A,\ a_{\sigma j\alpha}(z)\not\equiv 0\big\},
 \end{equation*}
 where $Q(a,b)$ denotes the second quadrant of $\RR^2$ translated by the vector $(a,b)$ (i.e. $Q(a,b):=\{(x,y)\colon x\leq a,\ y\geq b\}$ for any $(a,b)\in\RR^2$) and $\alpha\cdot s=\alpha_1s_1+\cdots+\alpha_Ns_N$ is the scalar product of $\alpha$ and $s$.
 \end{defi}

The definition given above is a fairly general one, covering a wide variety of operators. However, from now on we will focus exclusively on the case when coefficients of the moment operator $P$ do not depend on the variable $z$.

An example of a Newton polygon can be seen on Figure \ref{gfx:newton_polygon}.

\section{Formal norms}
In this section we introduce a useful tool, which allows us to keep together estimations of all $\Gamma_s$-moment derivatives of a given holomorphic function.
\begin{defi}
For $f(z)\in \Oo(D^{N}_R)$ let us define
the formal norm of $f$ by the formula: 
\begin{equation}
\|f(z)\|_{\rho}:=\sum_{\alpha\in\NN_0^N}\frac{|\partial_{\Gamma_{s},z}^{\alpha}f(z)|}{\Gamma(1+s\cdot\alpha)}\rho^{\alpha}\label{de:formal_norm}
\end{equation}
where $s=(s_1,\ldots,s_N)$, $\partial^\alpha_{\Gamma_s,z}=\partial^{\alpha_1}_{\Gamma_{s_1},z_1}\ldots\partial^{\alpha_N}_{\Gamma_{s_N},z_N}$,  $\rho=(\rho_1,\ldots,\rho_N)\in\CC^N$ and
$\rho^\alpha={\rho_1}^{\alpha_1}\cdots{\rho_N}^{\alpha_N}$.
\end{defi}

The above definition is a generalization
of a concept used in \cite{TY}. This section is devoted mainly to
presenting properties characterizing formal norms defined by (\ref{de:formal_norm}),
which in many cases are very similar to the results from \cite[Section 4]{TY}.

For any $a\geq0$ and $\rho\in\CC^N$ let us introduce the following formal power series: 
\begin{equation}
\label{eq:theta}
\Theta^{(a)}(\rho):=\sum_{\alpha\in\NN_0^N}\frac{\Gamma(1+s\cdot\alpha+a)}{\Gamma(1+s\cdot\alpha)}\rho^{\alpha}.
\end{equation}

They satisfy

\begin{lemma}
\label{le:4.3}
Let $b\geq 0$. Then for any $a>0$ we have
\begin{equation*}
\Theta^{(b)}(\rho)\ll e\left(\frac{e}{1+a+b}\right)^{a}\Theta^{(a+b)}(\rho).
\end{equation*}
\end{lemma} 

The proof of this fact is identical to
the one presented in \cite[Lemma 4.3.]{TY} and it will be omitted.

Using (\ref{eq:theta}) we may estimate the formal norm of holomorphic function in the following way:

\begin{lemma}
\label{le:1}
 If $f\in\Oo(D_R^N)$ then for every $r<r'<R$ there exists $h>0$
 such that
 $\sup_{z\in D^N_r}\|f(z)\|_{\rho}\ll C\Theta^{(0)}(h\rho)$,
 where $C=\sup_{z\in D^N_{r'}}|f(z)|$.
\end{lemma}

\begin{proof}
 It is sufficient to use (\ref{de:formal_norm}), (\ref{eq:theta}) and Proposition \ref{prop:general}. 
\end{proof}

\begin{lemma}
\label{le:4.2}
Let $f\in \Oo(D^N_R)$ be a function such that 
$\sup_{z\in D^N_r}\|f(z)\|_\rho\ll C\Theta^{(a)}(h\rho)$ for certain constants $r\in(0,R)$, $C,h>0$ and $a\geq 0$. Then for any $\beta\in\NN_0^N$ we have $\sup_{z\in D^N_r}\|\partial^{\beta}_{\Gamma_s}f(z)\|_\rho\ll Ch^{|\beta|}\Theta^{(a+s\cdot\beta)}(h\rho)$.
\end{lemma}
\begin{proof}
Since $\|f(z)\|_\rho\ll C\Theta^{(a)}(h\rho)$, for any $\alpha\in\NN_0^N$ and $z\in D_r^N$, it follows that
\begin{equation*}
\sup_{z\in D^N_r}\frac{|\partial_{\Gamma_s,z}^{\alpha}f(z)|}{\Gamma(1+s\cdot\alpha)}\leq\frac{Ch^{|\alpha|}\Gamma(1+s\cdot\alpha+a)}{\Gamma(1+s\cdot\alpha)},
\end{equation*}
which means that $\sup_{z\in D^N_r}|\partial_{\Gamma_s,z}^{\alpha}f(z)|\leq Ch^{|\alpha|}\Gamma(1+s\cdot\alpha+a)$. Seeing as the second inequality holds for all $\alpha\in\NN_0^N$, it is also true for $\bar{\alpha}=\alpha+\beta$. Hence, for all $\alpha\in\NN_0^N$ we have:
\begin{equation*}
\sup_{z\in D^N_r}\frac{|\partial_{\Gamma_s,z}^{\alpha+\beta}f(z)|}{\Gamma(1+s\cdot\alpha)}\leq\frac{Ch^{|\alpha|+|\beta|}\Gamma(1+s\cdot\alpha+s\cdot\beta+a)}{\Gamma(1+s\cdot\alpha)}.
\end{equation*}
\end{proof}

\begin{lemma}
\label{le:4.4}
Let us consider a function $f(t,z)=\sum_{n=0}^{\infty}f_{n}(z)t^{n}\in\Oo_{R}[[t]]_{s}$.
For any $r<R$ there exist constants $A,B,h>0$ such that for any $n\in\NN_{0}$
\begin{equation*}
\sup_{z\in D^N_r}\|f_{n}(z)\|_{\rho}\ll AB^n\Theta^{(0)}(h\rho)n!{}^{s}.
\end{equation*}
\end{lemma}

\begin{proof}
Since $f$ is of a Gevrey order $s$, for
every $r<r'<R$ there exist certain positive constants $A$ and $B$ such that we get
$\sup_{z\in D^N_{r'}}|f_{n}(z)|\leq AB^n n!^{s}$
for any $n\in\NN_{0}$. Using it and Lemma \ref{le:1}, we receive: 
\begin{equation*}
\sup_{z\in D^N_r}\|f_{n}(z)\|_{\rho}\ll
\sup_{z\in D^N_{r'}}|f_{n}(z)|\Theta^{(0)}(h\rho)\ll AB^{n}n!^{s}\Theta^{(0)}(h\rho).
\end{equation*}
\end{proof}

\section{Main problem}

\label{s:main_problem} 

Let us consider the moment differential operator
\begin{equation}
P(\partial_{m_{0},t},\partial_{m_{1},z_{1}},\ldots,\partial_{m_{N},z_{N}})=\partial_{m_{0},t}^{M}+\sum_{(j,\alpha)\in\Lambda}a_{j,\alpha}(t)\partial_{m_{0},t}^{j}\partial_{m,z}^{\alpha},\label{eq:diff_polynomial}
\end{equation}
where $m_0,m_1,\dots,m_N$ are moment functions of positive orders $s_0,s_1,\dots,s_N$ respectively, $m=(m_1,\dots,m_N)$, $s=(s_1,\dots,s_N)$, and $a_{j,\alpha}(t)$ are holomorphic functions in
a neighborhood of the origin.
Then the Newton polygon for (\ref{eq:diff_polynomial}) is given
by the set 
\begin{equation*}
N(P,s_0,s)=\textrm{conv}\left\{ Q(Ms_{0},-M)\cup\bigcup_{(j,\alpha)\in\Lambda}Q(s_{0}j+s\cdot\alpha,\,\ord(a_{j,\alpha})-j)\right\}.
\end{equation*}

For any $i=0,1,\dots,p+1$ let us denote the slope of the segment
$\Gamma_{i}$ by $k_{i}$ with all $k_{i}\geq 0$. It is easy to observe
that $0=k_{0}<k_{1}<\ldots<k_{p+1}=\infty$. If $p>0$ as well, then
we can also calculate the value of $k_{1}$ by the formula
\begin{equation*}
\frac{1}{k_{1}}=\max_{(j,\alpha)\in\Lambda}\left\{ \frac{s_{0}(j-M)+s\cdot\alpha}{\ord(a_{j,\alpha})-j+M}\right\}.
\end{equation*}
Hence, in the general case, for $q_{j,\alpha}:=\ord(a_{j,\alpha})-j+M$, we get
\begin{equation*}
\frac{1}{k_{1}}=\max\left\{ 0,\,\max_{(j,\alpha)\in\Lambda}\left\{ \frac{s_{0}(j-M)+s\cdot\alpha}{q_{j,\alpha}}\right\} \right\}.
\end{equation*}

Let us return to our main equation (\ref{eq:main}). We shall further
assume that: 
\begin{equation}
 \label{eq:cond_1}\phi_j\in\Oo_{\tilde{R}}\ (j=0,\dots,M-1)\ \textrm{and}\ f(t,z)\in\Oo_{\tilde{R}}[[t]]_{1/k_{1}}\ \textrm{for certain}\ \tilde{R}>0,
\end{equation}
\begin{equation}
 \label{eq:cond_2}
 \textrm{the set}\ \Lambda\subset\NN_{0}\times\NN_{0}^{N}\ \textrm{is finite},
\end{equation}
\begin{equation}
 \label{eq:cond_3}
 a_{j,\alpha}(t)\in\CC[[t]]_{1/k_{1}}\
\textrm{for any}\ (j,\alpha)\in\Lambda.
\end{equation}
We will also assume 
that 
\begin{equation}
\label{cond:orda}
\ord(a_{j,\alpha})\geq\max\{0,\,j-M+1\}\ \textrm{for any}\ (j,\alpha)\in\Lambda
\end{equation}
in order to ensure that (\ref{eq:main}) has a unique formal solution.

Additionally we assume that 
\begin{equation}
\label{eq:cond_5}
m_{0}\ \textrm{is a regular moment function}.
\end{equation}
We can then change the form of (\ref{eq:main}) using the composition of Borel transforms
of order $0$ with respect to $z_1,\dots,z_N$ given by: 
\begin{equation}
 \label{eq:composition}
\Bo_{\frac{\Gamma_{s}}{m},z}:=\Bo_{\frac{\Gamma_{s_{1}}}{m_{1}},z_{1}}\dots\Bo_{\frac{\Gamma_{s_{N}}}{m_{N}},z_{N}}.
\end{equation}
Since (\ref{eq:main}) is a linear equation with coefficients
which do not depend on $z$, by Proposition \ref{prop:commutation} we receive an equation equivalent to (\ref{eq:main}):
\begin{equation}
\left\{ \begin{aligned}\partial_{m_{0},t}^{M}v(t,z)+\sum_{(j,\alpha)\in\Lambda}a_{j,\alpha}(t)\partial_{m_{0},t}^{j}\partial_{\Gamma_{s},z}^{\alpha}v(t,z) & =g(t,z)\\
\partial_{m_{0},t}^{j}v(0,z) & =\psi_{j}(z)\textrm{ for }0\leq j<M
\end{aligned}
\right.,\label{eq:borel}
\end{equation}
where $v=\Bo_{\frac{\Gamma_{s}}{m},z}u$, $g=\Bo_{\frac{\Gamma_{s}}{m},z}f$ and $\psi_{j}=\Bo_{\frac{\Gamma_{s}}{m},z}\phi_{j}$
for $j=0,1,\dots,M-1$. By Remark \ref{rem:gevrey}
there exists $R>0$ such that
$\psi_j\in\Oo_{R}$ ($j=0,\dots,M-1$) and 
$g\in\Oo_{R}[[t]]_{1/k_{1}}$.

First we show that the formal solution of (\ref{eq:borel}) is of a Gevrey order $1/k_1$. To this end we will use the formal norms and their properties.

\begin{lemma}\label{lemma_bound} Let $\hat{v}(t,z)=\sum_{n=0}^{\infty}v_{n}(z)t^{n}$
be a formal solution of (\ref{eq:borel}). Then for every $r<R$ there exist constants $C,H,h>0$
such that for any $n\in\NN_0$ and $z\in D_{r}$ we have: 
\begin{equation}
\|v_{n}(z)\|_{\rho}\ll\frac{CH^{n}}{n!^{Ms_{0}}}\Theta^{(dn)}(h\rho),\label{eq:v_n-norm-bound}
\end{equation}
where $d=Ms_{0}+\frac{1}{k_{1}}.$
\end{lemma}

Before we can move on to the proof of this fact, we will present an additional technical lemma:
\begin{lemma}\label{lemma_factorial}
Let $M\in\NN$. Then the following formula holds true:
\begin{equation}
\frac{(n-M)!}{n!}\leq\left(\frac{M}{n}\right)^M\textrm{ for all }n\geq M\label{factorial_boundary}
\end{equation}
\end{lemma}
\begin{proof}
We shall prove this fact using induction with respect to $M$.

Firstly, let us assume that $M=1$. After substituting $1$ for $M$ in (\ref{factorial_boundary}) we receive an inequality
\begin{equation*}
\frac{1}{n}\leq\frac{1}{n},
\end{equation*}
which is obviously true for all $n\in\NN$.

Now suppose that
\begin{equation*}
\frac{(n-K)!}{n!}\leq\left(\frac{K}{n}\right)^K\textrm{ for all }n\geq K
\end{equation*}
holds true for any $K\leq M$. We shall prove that it is true for $K=M+1$ as well. Let us take $n\geq M+1$. Then:
\begin{equation*}
\frac{(n-M-1)!}{n!}=\frac{(n-M)!}{n!(n-M)}\leq\frac{1}{n-M}\left(\frac{M}{n}\right)^M.
\end{equation*}
It is enough to show that
\begin{equation*}
\frac{1}{n-M}\left(\frac{M}{n}\right)^M\leq\left(\frac{M+1}{n}\right)^M\frac{M}{n},
\end{equation*}
which is equivalent to the inequality
\begin{equation*}
n\left[M\left(\frac{M+1}{M}\right)^M-1\right]\geq M^2\left(\frac{M+1}{M}\right)^M.
\end{equation*}
Seeing as $n\geq M+1$, to prove the last formula it is enough to show that
\begin{equation*}
\frac{M^2\left(\frac{M+1}{M}\right)^M}{M\left(\frac{M+1}{M}\right)^M-1}\leq M+1.
\end{equation*}
This last inequality is equivalent to the following one:
\begin{equation*}
M\left(\frac{M+1}{M}\right)^M-M-1\geq 0,
\end{equation*}
which is true for all $M\geq 1$, because $\left(\frac{M+1}{M}\right)^M\geq\frac{M+1}{M}$.
\end{proof}

\begin{proof}[Proof of Lemma \ref{lemma_bound}]
For $n\leq M-1$ we have $v_{n}(z)=\frac{\psi_{n}(z)}{m_{0}(n)}$
and (\ref{eq:v_n-norm-bound}) holds. Assume then that (\ref{eq:v_n-norm-bound}) (with $n$ replaced by $i$)
is true for all $i\leq n-1$. We shall show the same for $n$. To
this end, first we note that if $v(t,z)=\sum_{n=0}^{\infty}v_{n}(z)t^{n}$
then: 
\begin{equation*}
\sum_{n=0}^{\infty}v_{n}(z)\partial_{m_{0},t}^{M}t^{n}+\sum_{(j,\alpha)\in\Lambda}a_{j,\alpha}(t)\sum_{n=0}^{\infty}(\partial_{\Gamma_{s},z}^{\alpha}v_{n}(z))\partial_{m_{0},t}^{j}t^{n}=g(t,z).
\end{equation*}
After differentiating with respect to $t$ and multiplying both sides
of the equation by $t^{M}$ we receive a formula: 
\begin{equation*}
\sum_{n=0}^{\infty}v_{n}(z)\frac{m_{0}(n)}{m_{0}(n-M)}t^{n}+\sum_{(j,\alpha)\in\Lambda}t^{M-j}a_{j,\alpha}(t)\sum_{n=0}^{\infty}\partial_{\Gamma_{s},z}^{\alpha}v_{n}(z)\frac{m_{0}(n)}{m_{0}(n-j)}t^{n}=t^{M}g(t,z).
\end{equation*}
Seeing as $t^{M}g(t,z)=\sum_{n=M}^{\infty}g_{n}(z)t^{n}$ and $t^{M-j}a_{j,\alpha}(t)=\sum_{p=q_{j,\alpha}}^{\infty}c_{j,\alpha,p}t^{p}$ for every $(j,\alpha)\in\Lambda$, we receive for all
$n\geq M$: 
\begin{equation}
\label{eq:v_n}
v_{n}(z)=\frac{m_{0}(n-M)}{m_{0}(n)}\left[g_{n}(z)-\sum_{(j,\alpha)\in\Lambda}\sum_{p=q_{j,\alpha}}^{n}c_{j,\alpha,p}\frac{m_{0}(n-p)}{m_{0}(n-p-j)}\partial_{\Gamma_{s},z}^{\alpha}v_{n-p}(z)\right],
\end{equation}
assuming the convention that whenever $n-p-j\leq0$, the term 
$
\frac{m_{0}(n-p)}{m_{0}(n-p-j)}
$
is equal to $0$. Please note that from (\ref{cond:orda}) it follows that $q_{j,\alpha}\geq 1$ for any $(j,\alpha)\in\Lambda$. Consequently $v_n$ given by (\ref{eq:v_n}) are coefficients of a unique formal solution of (\ref{eq:borel}).

Because  $a_{j,\alpha}(t)\in\CC[[t]]_{1/k_{1}}$ for all $(j,\alpha)\in\Lambda$,
there exist constants $A_{j,\alpha},B>0$ such that
$\left|c_{j,\alpha,p}\right|\leq A_{j,\alpha}B^{p}(p-q_{j,\alpha})!^{1/k_{1}}$
for any $p\geq q_{j,\alpha}$. Additionally, since $t^{M}g(t,z)\in\Oo_{R}[[t]]_{1/k_1}$, by Lemma \ref{le:4.4} we may
assume that there exist also constants $K,h>0$
such that $\|g_n(z)\|_{\rho}\ll KB^n\Theta^{(0)}(h\rho)n!^{1/k_1}$  for any $z\in D_r^N$ and $n\in\NN_0$.
So for any $z\in D_r^N$ we can estimate: 
\begin{multline*}
\|v_{n}(z)\|_{\rho} \ll\frac{(n-M)!^{s_{0}}}{a^{M}n!^{s_{0}}}\Big[KB^{n}\Theta^{(0)}(h\rho)n!^{1/k_{1}}+\\
\sum_{(j,\alpha)\in\Lambda}\sum_{p=q_{j,\alpha}}^{n}A_{j,\alpha}A^{j}B^p(p-q_{j,\alpha})!^{1/k_{1}}
\frac{(n-p)!^{s_{0}}}{(n-p-j)!^{s_{0}}}\|\partial_{\Gamma_s,z}^{\alpha}v_{n-p}(z)\|_\rho\Big].
\end{multline*}
By Lemma \ref{le:4.2} and by the inductive assumption we get
$$\|\partial_{\Gamma_s,z}^{\alpha}v_{n-p}(z)\|_\rho\ll
\frac{CH^{n-p}h^{|\alpha|}}{(n-p)!^{Ms_0}}\Theta^{(d(n-p)+s\cdot\alpha)}(h\rho)\quad \textrm{for any}\quad z\in D^N_r.$$ Hence, continuing our estimation we see that
\begin{multline*}
\|v_{n}(z)\|_{\rho} \ll\frac{(n-M)!^{s_{0}}}{a^{M}n!^{s_{0}}}\Big[KB^{n}\Theta^{(0)}(h\rho)n!^{1/k_{1}}+\\\sum_{(j,\alpha)\in\Lambda}\sum_{p=q_{j,\alpha}}^{n}\frac{A_{j,\alpha}A^{j}B^p(p-q_{j,\alpha})!^{1/k_{1}}(n-p)!^{s_{0}}CH^{n-p}h^{|\alpha|}}{(n-p-j)!^{s_{0}}(n-p)!^{Ms_{0}}}\Theta^{(d(n-p)+s\cdot\alpha)}(h\rho)\Big].
\end{multline*}

Using Lemma \ref{lemma_factorial} we receive: 
\begin{equation*}
\|v_{n}(z)\|_{\rho}\ll\frac{M^{Ms_{0}}}{a^{M}n^{Ms_{0}}}\left[KB^n\Theta^{(0)}(h\rho)n!^{1/k_{1}}+I\right],
\end{equation*}
where 
\begin{equation*}
I=\sum_{(j,\alpha)\in\Lambda}\sum_{p=q_{j,\alpha}}^{n}\frac{A_{j,\alpha}A^{j}B^p(p-q_{j,\alpha})!^{1/k_{1}}(n-p)!^{s_{0}}CH^{n-p}h^{|\alpha|}}{(n-p-j)!^{s_{0}}(n-p)!^{Ms_{0}}}\Theta^{(d(n-p)+s\cdot\alpha)}(h\rho).
\end{equation*}

Furthermore, let us note that by Lemma \ref{le:4.3}: 
\begin{align*}
\frac{KM^{Ms_{0}}B^n}{a^{M}n^{Ms_{0}}}
\Theta^{(0)}(h\rho)n!^{1/k_{1}} & =\frac{KM^{Ms_{0}}B^n}{a^{M}}\frac{(n^{s_{0}})^{Mn-M}}{(n^{s_{0}})^{Mn}}\Theta^{(0)}(h\rho)n!^{1/k_{1}}\\
 & \ll\frac{KM^{Ms_{0}}B^n(n^{n})^{Ms_{0}+1/k_{1}}}{a^{M}n!^{Ms_{0}}}\Theta^{(0)}(h\rho)\\
 & =\frac{KM^{Ms_{0}}B^n}{a^{M}n!^{Ms_{0}}}n^{dn}\Theta^{(0)}(h\rho)\\
 & \ll\frac{KM^{Ms_{0}}B^n}{a^{M}n!^{Ms_{0}}}e\left(\frac{en}{1+dn}\right)^{dn}\Theta^{(dn)}(h\rho)\\
 & \ll\frac{KM^{Ms_{0}}e}{a^{M}n!^{Ms_{0}}}\left(\frac{e^{d}B}{d^{d}}\right)^{n}\Theta^{(dn)}(h\rho).
\end{align*}

It is enough to choose $C\geq\frac{2KM^{Ms_{0}}e}{a^{M}}$ and
$H\geq\frac{e^{d}B}{d^{d}}$, for $\frac{1}{2}\frac{CH^{n}}{n!^{Ms_{0}}}\Theta^{(dn)}(h\rho)$
to be a majorant of $\frac{KM^{Ms_{0}}B^n}{a^{M}n^{Ms_{0}}}\Theta^{(0)}(h\rho)n!^{1/k_{1}}$.

Moreover, note that $\frac{(n-p)!}{(n-p-j)!}\leq n^{j}$ and $\frac{1}{(n-p)!}\leq\frac{n^{p}}{n!}$,
and because of these facts we receive the inequality: 
\begin{equation*}
\frac{(n-p)!^{s_{0}}}{(n-p-j)!^{s_{0}}}\frac{1}{(n-p)!^{Ms_{0}}}\leq\frac{n^{s_{0}(Mp+j)}}{n!^{Ms_{0}}}.
\end{equation*}
Analogously for $q_{j,\alpha}\leq p\leq n$ we estimate $(p-q_{j,\alpha})!^{1/k_1}\leq n^{(p-q_{j,\alpha})/k_{1}}$.

Using these inequalities and the fact that $s_{0}(j-M)\leq \frac{q_{j,\alpha}}{k_{1}}-s\cdot\alpha$
for any $(j,\alpha)\in\Lambda$, we conclude that: 
\begin{align*}
\frac{1}{n^{Ms_{0}}}I & \ll\frac{CH^{n}}{n!^{Ms_{0}}}\sum_{(j,\alpha)\in\Lambda}\sum_{p=q_{j,\alpha}}^{n}\frac{A_{j,\alpha}A^{j}h^{|\alpha|}B^{p}}{H^{p}}n^{s_{0}(Mp+j-M)+(p-q_{j,\alpha})/k_{1}}\\
 &\qquad\qquad\qquad\qquad\qquad\qquad\times\Theta^{(d(n-p)+s\cdot\alpha)}(h\rho)\\
 & \ll\frac{CH^{n}}{n!^{Ms_{0}}}\sum_{(j,\alpha)\in\Lambda}\sum_{p=q_{j,\alpha}}^{n}\frac{A_{j,\alpha}A^{j}h^{|\alpha|}B^p}{H^{p}}n^{dp-s\cdot\alpha}\Theta^{(d(n-p)+s\cdot\alpha)}(h\rho).
\end{align*}
Hence by Lemma \ref{le:4.3}: 
\begin{align*}
\frac{M^{Ms_{0}}}{a^{M}n^{Ms_{0}}}I & \ll\frac{CH^{n}}{n!^{Ms_{0}}}\sum_{(j,\alpha)\in\Lambda}\sum_{p=q_{j,\alpha}}^{n}\frac{A_{j,\alpha}A^{j}h^{|\alpha|}B^p}{a^{M}M^{-Ms_{0}}H^{p}}e\left(\frac{en}{1+dn}\right)^{dp-s\cdot\alpha}\Theta^{(dn)}(h\rho)\\
 & \ll\frac{CH^{n}}{n!^{Ms_{0}}}\sum_{(j,\alpha)\in\Lambda}\sum_{p=q_{j,\alpha}}^{\infty}\frac{eA_{j,\alpha}A^{j}h^{|\alpha|}}{a^{M}M^{-Ms_{0}}}\left(\frac{d}{e}\right)^{s\cdot\alpha}\left(\frac{e^{d}B}{Hd^{d}}\right)^{p}\Theta^{(dn)}(h\rho)\\
 & \ll\frac{CH^{n}}{n!^{Ms_{0}}}\Theta^{(dn)}(h\rho)\sum_{(j,\alpha)\in\Lambda}\frac{eA_{j,\alpha}A^{j}h^{|\alpha|}}{a^{M}M^{-Ms_{0}}}\left(\frac{d}{e}\right)^{s\cdot\alpha}\frac{\left(\frac{e^{d}B}{Hd^{d}}\right)^{q_{j,\alpha}}}{1-\frac{e^{d}B}{Hd^{d}}}.
\end{align*}
which can be bounded from
above by $\frac{1}{2}\frac{CH^{n}}{n!^{Ms_{0}}}\Theta^{(dn)}(h\rho)$
for sufficiently large $H$.
\end{proof}

\begin{prop}
\label{prop:estimation}
Let $\hat{v}(t,z)=\sum_{n=0}^{\infty}v_{n}(z)t^{n}$
be a formal solution of (\ref{eq:borel}). Then $\hat{v}(t,z)$ is of Gevrey order $1/k_1$ with respect to $t$, i.e. for any $r<R$ there exist constants
$\tilde{C},\tilde{H}>0$ such that 
\begin{equation*}
\sup_{z\in D_r^N}|v_{n}(z)|\leq \tilde{C}\tilde{H}^{n}n!^{1/k_{1}}\quad\textrm{for any}\quad n\in\NN_{0}.
\end{equation*}
\end{prop}
\begin{proof}
Let us note that $|v_{n}(z)|=\|v_{n}(z)\|_{0}$, where $0$ is the zero vector in $\CC^N$. 
Then we can use Lemma \ref{lemma_bound} to conclude that: 
\begin{align*}
|v_{n}(z)| & \leq\frac{CH^{n}}{n!^{Ms_{0}}}\Theta^{(dn)}(0)=\frac{CH^{n}}{n!^{Ms_{0}}}\Gamma(1+dn)\\
 & \leq\frac{CH^{n}}{n!^{Ms_{0}}}BL^{n}n!^{d}=\tilde{C}\tilde{H}^{n}n!^{1/k_{1}},
\end{align*}
which finishes the proof.
\end{proof}

We would like to find the similar result for the formal
solution $\hat{u}(t,z)$ of (\ref{eq:main}).
To this end we need the stronger version of Proposition \ref{prop:estimation}, where the series $v_n(z)=\sum_{|l|=0}^{\infty}v_{nl}z^l$, with $l=(l_1,\ldots,l_N)\in\NN_0^N$, is replaced by its majoring
series $M[v_n](z):=\sum_{|l|=0}^{\infty}|v_{nl}|z^l$.
\begin{prop}
 \label{pr:major}
 Let $\hat{v}(t,z)=\sum_{n=0}^{\infty}v_{n}(z)t^{n}$
be a formal solution of (\ref{eq:borel}). Then for every $r<R$ there exist constants $C,H>0$
such that: 
\begin{equation*}
\sup_{z\in D_r^N}|M[v_{n}](z)|\leq CH^{n}n!^{1/k_{1}}\quad\textrm{for any}\quad n\in\NN_{0},
\end{equation*}
where $d=Ms_{0}+\frac{1}{k_{1}}.$
\end{prop}
\begin{proof}
 It is sufficient to repeat the proofs of Lemma \ref{lemma_bound} and Proposition \ref{prop:estimation} with $\psi_n(z)$ replaced by
 $M[\psi_n](z)$ and $g_n(z)$ replaced by $M[g_n](z)$.
 Using the estimation
 \begin{multline*}
M[v_{n}](z)\ll\frac{m_{0}(n-M)}{m_{0}(n)}\Big[M[g_{n}](z)
\\
+\sum_{(j,\alpha)\in\Lambda}\sum_{p=q_{j,\alpha}}^{n}|c_{j,\alpha,p}|\frac{m_{0}(n-p)}{m_{0}(n-p-j)}\partial_{\Gamma_{s},z}^{\alpha}M[v_{n-p}](z)\Big]
\end{multline*}
 instead of (\ref{eq:v_n}) and repeating the proofs we get the assertion.
\end{proof}

Now we are ready to prove the main result of the paper:
\begin{thm}
Let $\hat{u}(t,z)=\sum_{n=0}^{\infty}u_{n}(z)t^{n}$
be a formal solution of (\ref{eq:main}). We also assume that the conditions (\ref{eq:cond_1}), (\ref{eq:cond_2}), (\ref{eq:cond_3}), (\ref{cond:orda}) and (\ref{eq:cond_5}) are satisfied. Then $\hat{u}(t,z)$ is of Gevrey order $1/k_1$ with respect to $t$, i.e. there exists $\tilde{R}>0$ such that for any $r<\tilde{R}$ there exist constants
$C,H>0$ satisfying 
\begin{equation*}
\sup_{z\in D_r^N}|u_{n}(z)|\leq CH^{n}n!^{1/k_{1}}\quad\textrm{for any}\quad n\in\NN_{0}.
\end{equation*}
\end{thm}
\begin{proof}
Since $\sum_{n=0}^{\infty}v_{n}(z)t^{n}=\hat{v}(t,z)=\Bo_{\frac{\Gamma_s}{m},z}\hat{u}(t,z)=
\sum_{n=0}^{\infty}\Bo_{\frac{\Gamma_s}{m},z}u_{n}(z)t^{n}$, where $\Bo_{\frac{\Gamma_s}{m},z}$ is the composition of
Borel transforms of order zero with respect to $(z_1,\dots,z_N)$ given by (\ref{eq:composition}), we see that $v_n(z)=\Bo_{\frac{\Gamma_s}{m},z}u_{n}(z)$ for every $n\in\NN_0$.
By (\ref{eq:gevrey}) there exists $B>0$ such that
$$
\frac{\Gamma_s(l)}{m(l)}=\frac{\Gamma_{s_1}(l_1)}{m_1(l_1)}\times\cdots\times\frac{\Gamma_{s_N}(l_N)}{m_N(l_N)} \leq B^{|l|}\quad\textrm{for every}\quad l\in\NN_0^N.
$$ 
Hence we get
$$
u_n(z)\ll M[u_n](z)\ll\sum_{|l|=0}^{\infty}\frac{|u_{nl}|B^{|l|}}{\frac{\Gamma_s(l)}{m(l)}}z^l= M[v_n](Bz)
$$
and using Proposition \ref{pr:major} we conclude that
there exist constants
$C,H>0$ such that 
$$
|u_n(z)|\leq |M[v_n](Bz)|\leq
CH^{n}n!^{1/k_{1}}
$$
for every $z\in D^N_{r}$ with $r<\frac{R}{B}$,
which gives the assertion with $\tilde{R}=\frac{R}{B}$.
\end{proof}

\end{document}